\documentclass[11pt,a4paper]{article}

\usepackage[T1]{fontenc}
\usepackage[utf8]{inputenc}
\usepackage[english]{babel}
\usepackage{color}
\usepackage{wrapfig}
\usepackage{placeins}
\usepackage{subfigure}
\usepackage{tabu}
\usepackage{fullpage}
\usepackage[squaren]{SIunits}
\usepackage{graphicx}
\usepackage{natbib}
\usepackage{tikz}
\usepackage{pgfplots}
\usepackage{epstopdf}
\usepackage{epsfig}
\usepackage{pifont}

\usepackage[colorlinks=true,bookmarks=false,linkcolor=blue,urlcolor=blue,citecolor=blue,breaklinks=true]{hyperref}
\usepackage{breakcites}
\usepackage{tikz}
\usepackage{tikz-qtree}
\usepackage{eurosym}
\usepackage{amsmath}
\numberwithin{equation}{section}
\usepackage{amssymb}
\usepackage{amsthm}
\usepackage{mathrsfs}
\usepackage{blkarray}
\usepackage{dsfont}
\usepackage{comment}
\usepackage{relsize}
\usepackage{booktabs}

\usepackage{rotating}
\newcommand{\minimize}{\textrm{minimize}}

\newcommand{\R}{\mathbb{R}}

\newcommand{\calM}{\mathcal{M}}

\newcommand{\calO}{\mathcal{O}}

\definecolor{purple}{rgb}{0.74, 0.2, 0.64}

\usepackage{makecell}

\newcommand{\D}{\mathrm{D}}

\newcommand{\calL}{\mathcal{L}}

\newcommand{\transpose}{^\top}
\newcommand{\T}{^\top}

\newcommand{\hess}{\nabla^2}

\newcommand{\Rm}{\mathbb{R}^m}
\newcommand{\Rn}{\mathbb{R}^n}

\newcommand{\Rmn}{\mathbb{R}^{m\times n}}

\newcommand{\rank}{\operatorname{rank}}

\newcommand{\I}{\mathrm{I}}
\newcommand{\sigmamin}{\sigma_\mathrm{min}}

\newcommand{\sigmabar}{\underline{\sigma}}

\usepackage[most]{tcolorbox}

\usepackage{bm}
\newcommand{\argmin}{\operatorname{argmin}}

\usepackage{algorithm}
\usepackage{algpseudocode}
\usepackage{todonotes}

\newcommand{\kplus}{_{k+1}}
\newcommand{\inv}{^{-1}}
\newcommand{\invv}[1]{^{-#1}}

\newcommand{\aref}[1]{\hyperref[#1]{A\ref{#1}}}
\newcommand{\norm}[1]{\left\|#1\right\|}

\usepackage{apxproof}
\newtheorem{theorem}			     {Theorem}	[section]
	
\newtheorem{lemma}	      [theorem]  {Lemma}		
\newtheorem{definition}	         {Definition}[section]
\newtheorem{assumption} {A\ignorespaces}
\newtheorem{remark}			     {Remark}	[section]
\newtheoremrep{example}			     {Example}
\newtheoremrep{proposition} [theorem] {Proposition}

\usepackage{listings}
\usepackage{textcomp}
\definecolor{listinggray}{gray}{0.9}
\definecolor{lbcolor}{rgb}{0.9,0.9,0.9}
\usepackage{mathtools}
\mathtoolsset{showonlyrefs=false}

\usepackage{authblk}

\usepackage{accents}
\newlength{\dhatheight}

\newcommand{\Flow}{F_{\mathrm{low}}}
\newcommand{\flow}{f_{\mathrm{low}}}
\newcommand{\Tlimit}{T(\varepsilon_0)}
\newcommand{\indexTlimit}{_{T(\varepsilon_0)}}
\newcommand{\LICQregion}{\mathcal{C}_{R}}
\newcommand{\sublevelset}{\calL_f(2f(x_0) - \flow)}

\newcommand{\betamax}{\beta_{\mathrm{max}}}

\newcommand{\conv}{\operatorname{conv}}

\newcommand{\Tsinglehat}{\hat{T}(\varepsilon_0)}
\newcommand{\Tdoublehat}{\tilde{T} (\varepsilon_0)}

\newcommand{\Lipf}{L_{f,0}}
\newcommand{\Lipgradf}{L_{f,1}}
\newcommand{\Liphessf}{L_{f,2}}
\newcommand{\LipgradP}{L_{\phi,1}}
\newcommand{\LiphessP}{L_{\phi,2}}
\newcommand{\xkp}{x_{k+1}}
\newcommand{\Omegastar}{\Omega_{*}}

\newcommand{\Deltaf}{(f(x_0)- \flow)}
\newcommand{\taucap}{\tau_{\mathrm{cap}}}
\newcommand{\tcalO}{\tilde{\mathcal{O}}}

\pgfplotsset{compat=1.17}

\title{Complexity of quadratic penalty methods with adaptive accuracy under a
PL condition for the constraints}
\author[1]{Florentin Goyens\thanks{Corresponding author: \url{florentin.goyens@uclouvain.be}. Funding for F.Goyens was provided by the Fonds de la Recherche
Scientifique  FNRS under Grant T.0001.23.}}
\author[1]{Geovani N. Grapiglia}

\affil[1]{ICTEAM Institute, UCLouvain, Louvain-la-Neuve, Belgium}

\date{\today}

\begin{document}
\maketitle

\begin{abstract}
We study the quadratic penalty method (QPM) for smooth nonconvex optimization
problems with equality constraints.
Assuming the constraint violation satisfies the PL condition near the feasible
set, we derive sharper worst-case
complexity bounds for obtaining approximate first-order KKT points. When the
objective and constraints
are twice continuously differentiable, we show that QPM equipped with a
suitable first-order inner solver
requires at most $\mathcal{O}(\varepsilon_{0}^{-1}\varepsilon_{1}^{-2})$
first-order oracle calls to
find an $(\varepsilon_{0},\varepsilon_{1})$-approximate KKT point---that is, a
point
that is $\varepsilon_{0}$-approximately feasible and
$\varepsilon_{1}$-approximately stationary.
Furthermore, when the objective and constraints are three times continuously
differentiable, we show that
QPM with a suitable second-order inner solver
requires at most
$\mathcal{O}\left(\varepsilon_{0}^{-1/2}\varepsilon_{1}^{-3/2}\right)$
second-order oracle calls to find an
$(\varepsilon_{0},\varepsilon_{1})$-approximate KKT point.
We also introduce an adaptive, feasibility-aware stopping criterion for the
subproblems, which relaxes
the stationarity tolerance when far from feasibility. This rule preserves all
theoretical guarantees while
substantially reducing computational effort in practice.\newline
\end{abstract}

\tableofcontents

\section{Introduction}

We consider the equality-constrained optimization problem
\begin{equation}\label{eq:P}\tag{P}
\underset{x \in \mathbb{R}^n}{\minimize} \;f(x) \textrm{ subject to } c(x) = 0,
\end{equation}
where $f\colon\mathbb{R}^n \to \mathbb{R}$ and $c\colon\mathbb{R}^n \to
\mathbb{R}^m$ are smooth and possibly nonconvex functions with $m \leq n$.
The idea of penalty methods is to transform~\eqref{eq:P} into a sequence of
unconstrained optimization problems. At its $k$th iteration, the Quadratic
Penalty Method (QPM) computes $\xkp$ by approximately minimizing the penalized
objective
\begin{equation}\label{eq:Q}
    Q_{\beta_k}(x) = f(x) + \dfrac{\beta_k}{2}\norm{c(x)}^2,
\end{equation}
where \(\beta_k > 0\) is a penalty parameter that determines the weight given
to the constraint violation. Each inner minimization (also called subproblem)
uses an increasingly larger value of $\beta_k$.
This raises the cost of infeasibility, thereby pushing iterates closer to the
feasible region. As \(\beta_k \to \infty\), minimizers of $Q_{\beta_k}$ ideally
approach feasible points of the constrained problem~\eqref{eq:P}. The quadratic
penalty method has a long history
\citep{zangwill1967NonLinear,bertsekas1997Nonlinear} and is an important
topic in nonconvex optimization, in part because it serves as a building block
for augmented Lagrangian methods.

Traditionally, the theoretical analysis of optimization methods has focused on
asymptotic convergence—that is, identifying conditions under which the iterates
converge (or have accumulation points converging) to a KKT point of
problem~\eqref{eq:P}. More recently, attention has shifted toward
\emph{worst-case complexity}, which aims to answer the following question:
\begin{quote}\emph{
How many iterations (and oracle calls) are required, in the worst case,
to generate an approximate KKT point?}
\end{quote}
Given positive tolerances \(\varepsilon_0\) and \( \varepsilon_1\), an \((\varepsilon_0,
\varepsilon_1)\)-approximate KKT point for~\eqref{eq:P} is a point $x\in \Rn$ such that 
there exists $\lambda\in \Rm$ satisfying
\begin{align}\label{eq:focp}\tag{$(\varepsilon_0,\varepsilon_1)$-KKT}
    \|c(x)\| &\leq \varepsilon_0, & \text{and} && 
    \| \nabla f(x) - \sum_{i=1}^m \lambda_i \nabla c_i(x)\| \leq \varepsilon_1.
\end{align}
Worst-case complexity results are typically expressed in terms of oracle calls. 
A call to a \emph{first-order oracle} returns the values of $f$, $c$, and their first-order derivatives at a given point $x\in \Rn$. 
A \emph{second-order oracle} additionally provides second-order derivatives. 
Higher-order oracles are also possible, but are not considered here. 

In the context of QPM, the \emph{outer iteration complexity} refers to the number
of subproblems solved (or equivalently, the number of updates of the penalty
parameter) before reaching an~\ref{eq:focp} point. At each outer iteration, an
unconstrained optimization method approximately minimizes the penalized
function \( Q_{\beta_k} \), using either a first- or second-order method, depending on the available oracle. 

The \emph{inner evaluation complexity} is the
number of oracle calls required to approximately minimize 
a single subproblem. 
These two notions come together in the
\emph{total evaluation complexity}, which is the total
number of oracle calls required to find
an~\ref{eq:focp} point.

Understanding and improving these bounds is central to the design of
optimization algorithms, as	the insight derived from the complexity analysis
may give clues to design methods with better practical performance.

\subsection{Contributions}

Our work improves complexity bounds of QPM for obtaining approximate
first-order KKT points of~\eqref{eq:P}.
We analyze a simple, practical version of QPM
(Algorithm~\ref{algo_quadratic_penalty}), without introducing any auxiliary
mechanisms designed to facilitate the theoretical analysis.

Our main contributions are as follows:
\begin{enumerate}
\item We introduce an adaptive \emph{feasibility-aware tolerance} for the
subproblems,  which relaxes the stationarity requirement when far from
feasibility, thus avoiding wasted computation in early iterations (Definition~\ref{def:feas_tol}). This
substantially improves practical performance.
\item We improve the \emph{outer iteration complexity} bound of QPM.
If the constraint violation satisfies the Polyak–Łojasiewicz (PL) condition
near the feasible set~(Assumption~\aref{assu_LICQ}), we show that $\beta_k \geq
\calO(\varepsilon_0\inv)$ ensures $\norm{c(x_{k+1})}\leq \varepsilon_0$
(Lemma~\ref{lemma_tight_upper_bound_beta}), thus improving on the existing
$\beta_k \geq \calO(\varepsilon_0\invv{2})$
bound~\citep{grapiglia2023Worstcase}. 
The PL condition is a mainstream assumption in constrained optimization,
appearing under various names. In much of the literature, it is assumed
to hold globally or at all iterates; here, we only require it to hold 
 in a neighborhood of the feasible set.
\item We establish \emph{total evaluation complexity} bounds for QPM.
Under~\aref{assu_LICQ}, when the objective and constraints are twice
continuously differentiable, we show that QPM equipped with a suitable
first-order inner solver requires at most
$\tcalO(\varepsilon_{0}^{-1}\varepsilon_{1}^{-2})$ first-order oracle calls to
find an~\ref{eq:focp} point. Furthermore, under~\aref{assu_LICQ}, when the
objective and constraints are three times continuously differentiable, we show
that QPM with a suitable second-order inner solver achieves an~\ref{eq:focp}
point in at most
$\tcalO\left(\varepsilon_{0}^{-1/2}\varepsilon_{1}^{-3/2}\right)$ second-order
oracle calls.
We do not assume Lipschitz continuity of the derivatives of $f$ and $c$ on a
set that contains the iterates and trial points, which may be difficult to ensure a priori. Instead, we only assume that $f$ has a bounded
sublevel set (\aref{assu_bounded_sublevelsets}).
 \end{enumerate}
The table below summarizes the complexity results
established in this paper.
    
    \begin{table}[!ht]
\centering
\renewcommand{\arraystretch}{1.2}
\begin{tabular}{|l|c|c|}
\hline
QPM: target~\ref{eq:focp} & Without PL & Under PL (\aref{assu_LICQ}) \\
\hline
Penalty parameter $\beta_{\max}$
& $\displaystyle \mathcal{O}\!\left(\varepsilon_0^{-2}\right)^*$
& $\displaystyle \mathcal{O}\! \left(\varepsilon_0^{-1}\right)$ \\[1.0ex]
\hline
Outer iteration complexity
&$\mathcal{O}\!\left(\log\!\left(\beta_0\inv\varepsilon_0^{-2}\right)\right)^*$
& $\mathcal{O}\!\left(\log\!\left(\beta_0\inv\varepsilon_0^{-1}\right)\right)$
\\[1.0ex]
\hline
\makecell{Total evaluation complexity\\  with first-order inner solver
(\aref{assu_A1})}
& $\displaystyle \calO \!\left(\log\!	\left(\beta_0\inv
\varepsilon_0^{-2}\right)	\varepsilon_0^{-2}\,\varepsilon_1^{-2}\right)^*$
& $\displaystyle \calO \!\left(\log\!	\left(\beta_0\inv \varepsilon_0^{-1}\right)
\varepsilon_0^{-1}\,\varepsilon_1^{-2}\right)$ \\[0.5ex]
\hline
\makecell{Total evaluation complexity\\  with second-order inner solver
(\aref{assu_A2})}
& $\displaystyle \calO \!\left(\log\!	\left(\beta_0\inv
\varepsilon_0^{-2}\right)	\varepsilon_0^{-1}\,\varepsilon_1^{-3/2}\right)$
& $\displaystyle \calO \! \left(\log\!	\left(\beta_0\inv \varepsilon_0^{-1}\right)
\varepsilon_0^{-1/2}\,\varepsilon_1^{-3/2}\right)$ \\[0.5ex]

\hline
\end{tabular}
\caption{Summary of our complexity results for the quadratic penalty method,
with and without the PL condition on the constraint violation. $^*$Already appears in~\citep{grapiglia2023Worstcase} for QPM with constant subproblem tolerance, i.e., $\tau(x) \equiv \varepsilon_1$ in~\eqref{eq:first_order_condition}.}
\label{tab:intro_complexity_summary}
\end{table}

\FloatBarrier

\subsection{Related literature}

The worst-case evaluation complexity of methods for constrained nonconvex
problems
has been widely studied in recent years~\citep{cartis2022evaluation}.
For first-order schemes,~\citet{cartis2011evaluation}
analyzed an exact penalty method,
showing a bound of at most $\mathcal{O}(\varepsilon^{-5})$ problem evaluations
to reach an
$\varepsilon$-KKT point, which corresponds to an~\ref{eq:focp} point with
$\varepsilon=\varepsilon_{0}=\varepsilon_1$ in our definition. This rate
improves to $\mathcal{O}(\varepsilon^{-2})$
under the assumption that the penalty parameters remain bounded.
Comparable $\mathcal{O}(\varepsilon^{-2})$ bounds were obtained for a
two-phase method~\citep{cartis2014Complexity}, inexact restoration
methods~\citep{bueno2020Complexity}, and using Fletcher's augmented
Lagrangian~\citep{goyens2024computing}. 
For sequential quadratic programming schemes, 
\citet{facchinei2021Ghost} established 
an $\mathcal{O}(\varepsilon^{-4})$ bound
and~\citet{curtis2024Worstcase} later showed 
an $\calO(\varepsilon\invv{2})$
bound. 

Second-order methods offer improved complexity:
adaptive cubic regularization and trust-funnel schemes attain a worst-case rate
of
$\mathcal{O}(\varepsilon^{-3/2})$
\citep{cartisEvaluationComplexityCubic2013,curtis2018Complexity},
while higher-order models achieve
$\mathcal{O}(\varepsilon^{-(p+1)/p})$ for $p$-th order methods ($p\ge2$),
as shown
in~\citep{birgin2016Evaluation,martinez2017Highorder,cartis2022evaluation}.

Inexact penalty methods, such as the quadratic penalty and augmented Lagrangian
methods have also received a lot of attention. Under a non-singularity
condition on the constraints, \citet{li2021Rateimproved} showed an $\tilde
\calO(\varepsilon\invv{3})$ bound for an inexact augmented Lagrangian method.

\citet{birgin2020Complexity} give complexity results for the \textsc{Algencan}
algorithm~\citep{andreani2008augmented}, an augmented Lagrangian method
designed for practical performance. They show that the number of outer
iteration is logarithmic in several circumstances, and combine it with an inner
solver of arbitrary complexity.
\citet{grapiglia2023Worstcase} shows that the quadratic penalty method
(Algorithm~\ref{algo_quadratic_penalty}) with non-adaptive subproblem accuracy
$\tau(x) \equiv \varepsilon_1$ and gradient descent with an Armijo linesearch
in the subproblems requires at most $\tcalO(\varepsilon^{-4})$ problem
evaluations to find an $\varepsilon$-KKT point.
For $f$ and $c$ weakly convex,~\citet{lin2020inexact} show an $\tilde
\calO(\varepsilon\invv{3})$ bound for a quadratic penalty method.

For linearly constrained problems,~\citet{kong2019Complexity}
proved an $\mathcal{O}(\varepsilon^{-3})$ bound for a quadratic penalty method
with an accelerated
first-order inner solver. Under a Slater-type condition, proximal-point or
augmented
Lagrangian variants reach $\tilde{\mathcal{O}}(\varepsilon^{-5/2})$,
e.g.,~\citep{li2021augmented,lin2020inexact,melo2020Iterationcomplexity}.
\citet{grapiglia2021complexity} obtained a
$\tilde{\mathcal{O}}(\varepsilon^{-(p+1)/p})$ bound for an augmented Lagrangian
method with a $p$-th order inner solver, matching the order of the
proximal augmented Lagrangian bound of~\citet{xie2021complexity}.
\citet{he2023NewtonCG} later improved this bound, giving a result of $\tilde
\calO(\varepsilon\invv{7/2})$ under an LICQ condition and $\tilde
\calO(\varepsilon\invv{11/2})$ without LICQ.

Outside the realm of penalty methods, Riemannian optimization methods are
designed for instances of~\eqref{eq:P} where the feasible set is a smooth
manifold, with a convenient expression for the tangent space and projection
onto the feasible set.
Under a Lipschitz continuity assumption, Riemannian gradient descend finds a
$(0,\varepsilon_1)$-KKT point in at most $\calO(\varepsilon_1\invv{2})$
iterations~\citep{boumal2019global}. A similar bound holds for a second-order
Riemannian trust-region, but is improved to $\calO(\log \log(\varepsilon\inv))$
for a class of functions with strict saddle points~\citep{goyens2025riemannian}.

Overall, comparing these results is nontrivial, as they address
different problem classes, and usually consider different
 regularity assumptions on $c$ or requirements on the
quality of the initial point. The present work improves the existing
literature by establishing complexity bounds for a quadratic penalty method that allows inexact
subproblem solutions.

\subsection{Contents}
The paper is organized as follows. In Section~\ref{sec_qpm}, we define the
Quadratic Penalty Method (QPM) and introduce the feasibility-aware tolerance. In Section~\ref{sec_outer_noLICQ}, we show that
the outer iteration complexity result from~\citep{grapiglia2023Worstcase} still
applies when the accuracy in the subproblems uses the 
adaptive
tolerance in~\eqref{eq:first_order_condition}. In Section~\ref{sec_outer_LICQ}, we
derive an improved outer iteration complexity bound under a PL condition on the
constraint violation (\aref{assu_LICQ}). In
Section~\ref{sec_total_complexity_FO}, we give total evaluation complexity
bounds for QPM when a first-order method performs the inner minimization.
In Section~\ref{sec_total_complexity_SO}, we give total evaluation complexity
bounds when a second-order method performs the inner minimization. In
Section~\ref{sec_numerics}, we illustrate our findings empirically and show
that the adaptive accuracy in the subproblems improves practical performance.

\subsection*{Notations}
Let $\norm{\cdot}$ to denote the usual $2$-norm in $\Rn$. Let $J(x) \in \Rmn$
denote the Jacobian of the function $c$, and $\mathbb{N}^*$ denote the set of
positive natural numbers. We use $\conv(A)$ to denote the convex hull of a set
$A$. 

\section{The Quadratic Penalty Method}\label{sec_qpm}
We consider the Quadratic Penalty Method described in
Algorithm~\ref{algo_quadratic_penalty}.
A classical result about the quadratic penalty
method~\citep[Prop.~4.2.1]{bertsekas1997Nonlinear} shows that, under mild
assumptions, global minimizers of $Q_{\beta_k}$ converge to the global
minimizer of \eqref{eq:P} as $\beta_k \to \infty$. This is of little practical
value since it is undesirable to have $\beta_k$ grow arbitrarily large, as it
makes the subproblems badly conditioned. Furthermore, it is in general not
possible to find a global minimizer of the quadratic penalty. This raises two
important questions:
 \begin{itemize}
 	\item[(i)] How rapidly should $\beta_k$ grow?
 	\item[(ii)] To what accuracy should each subproblem be solved?
 \end{itemize}
At iteration $k$, we update the penalty parameter as $\beta_{k+1} = \alpha
\beta_k$, where $\alpha>1$ is a parameter of QPM.
We require that the approximate minimizer of $Q_{\beta_k}$ yields a reduction
in the value of $Q_{\beta_k}$ compared to the current point $x_k$ and the
initial point $x_0$~\eqref{eq:zeroth_order_condition}; and we also require that
it be approximately first-order stationary~\eqref{eq:first_order_condition}.
 
In~\citep{grapiglia2023Worstcase}, at each iteration, the subproblem is solved
up to the target accuracy~$\varepsilon_1$. However, it is customary in
optimization softwares to use a tolerance larger than $\varepsilon_1$ for the
stationarity in the early iterations. Various heuristics  are used in practice,
such as starting with $\sqrt{\varepsilon_1}$ and dividing the tolerance by a
factor $10$ at each outer iteration
(\textsc{Algencan},\,\citet{andreani2008augmented}); another strategy is given
in~\citep{eckstein2013Practical}. We propose an adaptive and practical
tolerance on the gradient norm. The accuracy is proportional to the current
level of feasibility, which is a natural way to spare computational resources
when the iterates are far from the feasible set.
 
 \begin{definition}[Feasibility-aware tolerance]\label{def:feas_tol}
Given tolerances $\varepsilon_0>0,\varepsilon_1>0$, a \emph{feasibility-aware
tolerance} is a function $\tau:\mathbb{R}^n\to\mathbb{R}_+$ given, for some
$\tau_{\mathrm{cap}}\ge\varepsilon_1$ by
\begin{equation}\label{eq:tau_def_lemma}
    \tau(x)
    :=
    \max\Bigl\{
        \varepsilon_1,\
\min\Bigl(\tau_{\mathrm{cap}},\,\tfrac{\varepsilon_1}{\varepsilon_0}\,\norm{c(x)}\Bigr)
    \Bigr\}.
\end{equation}
\end{definition}
This ensures
\begin{equation}\label{eq:property_tau}
    \tau(x) \ge \varepsilon_1 \quad \text{for all } x\in\mathbb{R}^n,
    \qquad\text{and}\qquad
    \tau(x) = \varepsilon_1 \;\text{whenever } \|c(x)\|\le \varepsilon_0.
\end{equation}
Consequently, as soon as QPM generates $\xkp$ satisfying
 $\norm{c(x_{k+1})} \leq \varepsilon_0$, it automatically follows that $\norm{\nabla Q_{\beta_k}(x_{k+1})} \leq \varepsilon_1$, which
makes it an~\ref{eq:focp} point. In our numerical experiments
(Section~\ref{sec_numerics}), we compare the choice
  \begin{equation*}
\tau(x) :=
\max\left(\varepsilon_1,\dfrac{\varepsilon_1}{\varepsilon_0}\norm{c(x)}\right),
 \end{equation*}
corresponding to $\taucap = +\infty$; with the non-adaptive choice $\tau(x)
\equiv \varepsilon_1$, given by $\tau_{\mathrm{cap}} = 0$.
 
 \begin{algorithm}
\caption{Quadratic Penalty Method (QPM)}\label{algo_quadratic_penalty}
\begin{algorithmic}[1]
\State \textbf{Given:}  $\varepsilon_0>0$, $\varepsilon_1 >0$, $\alpha>0$,
$\beta_0 \geq 1$, feasibility-aware tolerance $\tau\colon \Rn \to \R_+$, and
$x_0\in \Rn$
\State $k \leftarrow 0$
\While{true}
\State Find $x_{k+1}\in \Rn$ an approximate minimizer of $Q_{\beta_k}$ that
satisfies
\begin{align}\label{eq:zeroth_order_condition}
Q_{\beta_k}(x_{k+1}) \leq \min\left\lbrace Q_{\beta_k}(x_k),
Q_{\beta_k}(x_0)\right\rbrace
\end{align}
and
\begin{align}\label{eq:first_order_condition}
	\norm{\nabla Q_{\beta_k}(x_{k+1})} \leq \tau(\xkp).
\end{align}
\If{$\norm{c(x_{k+1})} \leq \varepsilon_0$}
\State Return $x_{k+1}$ \Comment{$(\varepsilon_0, \varepsilon_1)$-KKT point
with multipliers $-\beta_k c(x_{k+1})$}
\EndIf
\State $\beta_{k+1} = \alpha \beta_k$.
\State $k \leftarrow k +1$
\EndWhile
\end{algorithmic}
\end{algorithm}

\FloatBarrier


 \section{Outer iteration complexity without the PL condition}
 \label{sec_outer_noLICQ}
This section establishes an outer iteration complexity bound for QPM
(Algorithm~\ref{algo_quadratic_penalty}). The main result of this section
already appears in~\citep{grapiglia2023Worstcase} with the choice $\tau(x)
\equiv \varepsilon_1$. We extend the analysis to include an 
adaptive tolerance in~\eqref{eq:first_order_condition}.

This section makes no regularity assumption on the Jacobian of the constraints.
Section \ref{sec_outer_LICQ} shows improved outer iteration complexity
guarantees under the assumption that the constraint violation satisfies the PL
condition.

We state standard assumptions.
 \begin{assumption}\label{assu_smooth_problem}
     The functions $f$ and $c$ are differentiable. 
 \end{assumption}
\begin{assumption}\label{assu_lower_bound_f}
There exists $\flow \in \R$ such that $f(x) \geq f_{low}$ for all $x\in \Rn$.
\end{assumption}
\begin{assumption}\label{assu_x0_eps_feasible}
The initial iterate $x_0\in \Rn$ is such that $\norm{c(x_0)} \leq
\dfrac{\varepsilon_0}{\sqrt{2}}$.
\end{assumption}

The following lemma states (in contrapositive form) that if, at some iteration
$k$, the penalty parameter $\beta_k$ is sufficiently large, then QPM 
terminates. Specifically, if
$\beta_k \geq \calO(\varepsilon_0\invv{2})$, then
$\norm{c(\xkp)}\leq \varepsilon_0$.
\begin{lemma}[\citet{grapiglia2023Worstcase},\,Corollary 3.3]
\label{lemma_loose_upper_bound_beta}
Under~\aref{assu_smooth_problem},~\aref{assu_lower_bound_f},
~\aref{assu_x0_eps_feasible}, if $\norm{c(x_k)}>\varepsilon_0$ for
$k=1,\dots,T-1$, then
 \begin{align}\label{eq_first_beta_upper_bound}
\beta_{T-2} < 4(f(x_0) - \flow) \varepsilon_0^{-2}.
 \end{align}
 \end{lemma}

Recall that $\mathbb{N}^*$ denotes the set of positive natural numbers. We now
show that the number of outer iterations of QPM is at most a logarithmic
factor of the upper bound on the penalty parameter given in
Lemma~\ref{lemma_loose_upper_bound_beta}.
\begin{theorem}[Outer iterations of QPM without the PL
condition]\label{thm_outer_complexity_no_licq}
Under~\aref{assu_smooth_problem},~\aref{assu_lower_bound_f},
~\aref{assu_x0_eps_feasible}, define
\begin{align}\label{eq_T_eps0}
T(\varepsilon_0):= \inf \{ k\in \mathbb{N}^*\colon \norm{c(x_k)} \leq
\varepsilon_0\}.
\end{align}
Then, 
\begin{align}\label{eq:Tsinglehat}
\Tlimit < \Tsinglehat:= 2 + \log_\alpha\left(
4(f(x_0)-\flow)\beta_0\inv\varepsilon_0^{-2}\right),
\end{align}
and the point $x_{T(\varepsilon_0)}$ generated by
Algorithm~\ref{algo_quadratic_penalty} is an~\ref{eq:focp} point
of~\eqref{eq:P}, satisfying
\begin{align}
\norm{c\!\left(x_{T(\varepsilon_0)}\right)}&\leq \varepsilon_0  &&\textrm{ and } &
\norm{\nabla f \!\left(x\indexTlimit \right) + \beta_{\Tlimit-1} \sum_{i=1}^m
c_i\!\left(x\indexTlimit\right) \nabla c_i \!\left(x\indexTlimit\right) } \leq \varepsilon_1 .
\end{align}
\end{theorem}
\begin{proof}
By definition of $\Tlimit$, we have $\norm{c(x_k)}>\varepsilon_0$ for
$k=1,\dots,\Tlimit-1$. By Lemma~\ref{lemma_loose_upper_bound_beta}, this gives
    \begin{align*}
\alpha^{\Tlimit-2} \beta_0  = \beta_{\Tlimit-2} < 4(f(x_0)-\flow)\varepsilon_0^{-2}.
    \end{align*}
    Therefore, $\Tlimit$ is finite and
 \begin{align*}
\Tlimit < 2 +
\log_\alpha\left(4(f(x_0)-\flow)\beta_0\inv\varepsilon_0^{-2}\right)=:
\Tsinglehat.
    \end{align*}
Finally, we have $\norm{c\! \left(x_{T(\varepsilon_0)}\right)}\leq \varepsilon_0$ by
construction, which by~\eqref{eq:first_order_condition} ensures
    \begin{align*}
\norm{\nabla Q_{\beta_{\Tlimit-1}} \!\left(x\indexTlimit\right)} &\leq \tau\!\left(x\indexTlimit\right) =
\varepsilon_1.
    \end{align*}
    To conclude, the identity
    \begin{align*}
\norm{\nabla Q_{\beta_{\Tlimit-1}}\!\left(x\indexTlimit\right)} = \norm{\nabla f
\!\left(x\indexTlimit\right) + \beta_{\Tlimit-1} \sum_{i=1}^m 
c_i\!\left(x\indexTlimit\right) \nabla c_i
\!\left(x\indexTlimit\right) },
    \end{align*}
implies that $x\indexTlimit$ is an~\ref{eq:focp} point with Lagrange
multipliers $-\beta_{\Tlimit-1} c \!\left(x\indexTlimit \right)$.
\end{proof}

In the next section, we make an additional assumption---that the constraint
violation satisfies the PL condition near the feasible set. This allows to
improve the complexity result.

     \section{Outer iteration complexity under the PL condition}
     \label{sec_outer_LICQ}
In this section, we show a new complexity result for QPM, which improves
the bound of Theorem~\ref{thm_outer_complexity_no_licq}. The analysis of the
previous section merely uses the decrease in the value of the penalty at each
subproblem~\eqref{eq:zeroth_order_condition}. To obtain a sharper bound, we
leverage the approximate stationarity
condition~\eqref{eq:first_order_condition}.
The most natural way to do so is to work under an additional regularity
condition on the Jacobian of the constraints.
     
      \begin{assumption}[PL condition for the constraint violation]\label{assu_LICQ}
There exists positive constants $R>\varepsilon_0$ and $\sigmamin$ such that for
all $x$ in the set
    \begin{align}
        \LICQregion := \{x\in \Rn\colon \norm{c(x)}\leq R\},
    \end{align}
     we have
\begin{align}\label{eq:PL}
         \norm{J(x)\transpose c(x)} \geq \sigmamin \norm{c(x)},
     \end{align}
     where $J(x) \in \R^{m\times n}$ denotes the Jacobian of $c$.
 \end{assumption}
Note that~\eqref{eq:PL} is the Polyak–Łojasiewicz (PL) condition for the
constraint violation $$\phi(x)=\tfrac12\|c(x)\|^2.$$
Indeed,~\eqref{eq:PL} is equivalent to
\begin{equation}
\phi(x) \leq \dfrac{1}{2 \sigmamin^2}\norm{\nabla \phi(x)}^2\quad \text{ for all }
x\in \LICQregion.
\end{equation}
It is weaker than the usual LICQ condition at a point $x$ which demands that
$J(x)$ be full rank, but we require a positive lower bound $\sigmamin>0$ in a
tubular neighbourhood of the feasible set. 
The regularity condition~\eqref{eq:PL} is relatively common in the literature, and appears under various names. Most analyses assume that the condition holds on $\Rn$, which
excludes many common applications. Notably, we derive a global bound
from a local condition.

 We also require that a sublevel set
of $f$ be bounded.
      \begin{assumption}\label{assu_bounded_sublevelsets}
         The set $\mathcal{L}_f\left(2f(x_0)-\flow\right)$ is bounded.
     \end{assumption}
 First, we provide several examples which satisfy our assumptions. In any of the examples below, taking $f$ coercive (i.e., $\underset{\norm{x}\to \infty}{\lim}
f(x) = +\infty$) is a sufficient condition to satisfy~\aref{assu_bounded_sublevelsets}.
 \begin{example}[Affine constraints]
 The problem
    \begin{equation}
\begin{aligned}
& \underset{x\in \Rn}{\min}
& & f(x)\\
& \text{subject to}
& & Ax=b,
\end{aligned}
\end{equation}
with $\rank(A)=m$ satisfies~\aref{assu_LICQ}.
 \end{example}
\begin{example}[Orthogonality constraints]
\label{example:stiefel}
The problem
        \begin{equation}
\begin{aligned}
& \underset{X\in \R^{n\times p}}{\min}
& & f(X)\\
& \text{subject to}
& & X\T X = \I_p,
\end{aligned}
\end{equation}
satisfies~\aref{assu_LICQ} for any
$\varepsilon_0 <R<1$ and $\sigmamin \leq 2\sqrt{1-R}$.
\end{example}
\begin{proof}
See~\citep[p.4]{goyens2024computing}.
\end{proof}

\begin{example}[Binary constraints]\label{example_binary}
The problem
\begin{equation}
\begin{aligned}
& \underset{x\in\Rn}{\min}
& & f(x)\\
& \text{subject to}
& & x_i(1-x_i) = 0, \quad \text{ for } i = 1,\dots,n,
\end{aligned}
\end{equation}
which encodes the binary
constraints $x_i\in\{0,1\}$, satisfies~\aref{assu_LICQ} for
 any $R$ such that $\varepsilon_0 < R <
\tfrac{1}{4}$ and \( \sigmamin \;=\; \sqrt{\,1-4R\,}\;>\;0\).
\end{example}
\begin{proof}
The Jacobian of $c$ is diagonal,
\(
J(x) = \mathrm{diag}\big(1-2x_1,\dots,1-2x_n\big).
\)
Using the identity
\[
(1-2x_i)^2 = 1 - 4c_i(x),
\]
for any $x$ such that $\|c(x)\|\le R$, we obtain
\begin{equation*}
\begin{aligned}
\|J(x)^\top c(x)\|^2
&= \sum_{i=1}^n (1-2x_i)^2\,c_i(x)^2\\
&= \sum_{i=1}^n (1-4c_i)\,c_i(x)^2\\
&\ge (1-4R)\sum_{i=1}^n c_i(x)^2\\
&= (1-4R)\|c(x)\|^2.
\end{aligned}
\end{equation*}
Thus,
\[
\|J(x)^\top c(x)\| \ge \sqrt{\,1-4R\,}\|c(x)\|\qquad\text{for all
}x\in\LICQregion.\qedhere
\]
\end{proof}

Furthermore, combining several nonsingular constraints gives a constraint
function that satisfies~\aref{assu_LICQ}.
\begin{proposition}[\citet{goyens2024computing}, \emph{Prop.1.1}]
\label{prop:product_manifold}
For $i=1,2,\dots,k$, consider $k$ functions $c_i\colon \Rn \to \R^{m_i}$ that
satisfy~\aref{assu_LICQ} with constants $R_i$ and $\sigmabar_i$.
Then, the function $c\colon \Rn \to \Rm$ with  $m=m_1 + \cdots +m_k$ defined by
$c(x) = \left(c_1(x), \dots, c_k(x)\right)\transpose$
satisfies~\aref{assu_LICQ} with constants $R = \min(R_1, \dots, R_k)$ and
$\sigmabar = \min(\sigmabar_1, \dots,\sigmabar_k)$.
\end{proposition}

We now turn to the complexity analysis under~\aref{assu_LICQ}, and begin with a
lemma stating that the iterates of QPM are contained in a sublevel set of
the function $f$.
 \begin{lemma}\label{lemma_sublevelset}
Under~\aref{assu_smooth_problem},~\aref{assu_lower_bound_f},
~\aref{assu_x0_eps_feasible}, 
if $\norm{c(x_k)}>\varepsilon_0$ for $k=1,\dots,T-1$, then
     \begin{align}
         x_k \in \sublevelset:= \left\lbrace
x\in \Rn: f(x) \leq 2 f(x_0) - \flow
         \right\rbrace,
     \end{align}
for $k=0,\dots,T-1$.
     \end{lemma}
     \begin{proof}
Clearly $x_0 \in \mathcal{L}_f \!\left(2f(x_0)-\flow\right)$. Let $k\in
\{1,\dots,T-1\}$. The decrease condition $Q_{\beta_{k-1}}(x_k) \leq
Q_{\beta_{k-1}} (x_0)$~\eqref{eq:zeroth_order_condition}, combined with
Lemma~\ref{lemma_loose_upper_bound_beta} and $\norm{c(x_0)} \leq
\frac{\varepsilon_0}{\sqrt{2}}$ gives
\begin{equation*}
         \begin{aligned}
             f(x_k) &\leq f(x_k) + \frac{\beta_{k-1}}{2} \norm{c(x_k)}^2\\ 
             &\leq f(x_0) + \frac{\beta_{k-1}}{2} \norm{c(x_0)}^2\\
&\leq f(x_0) + 2 (f(x_0) - \flow)\varepsilon_0^{-2}\frac{\varepsilon_0^{2}}2\\
             &= 2f(x_0) - \flow.
         \end{aligned}
         \end{equation*}
     \end{proof}

The next lemma, shows that if $\beta_k \geq \calO\!\left(\varepsilon_0\inv\right)$, then
QPM terminates, i.e., $\norm{c(\xkp)}\leq\varepsilon_0$.

\begin{lemma}\label{lemma_tight_upper_bound_beta}
Under~\aref{assu_smooth_problem},~\aref{assu_lower_bound_f},
~\aref{assu_x0_eps_feasible},~\aref{assu_LICQ},
~\aref{assu_bounded_sublevelsets}, if
$\norm{c(x_k)}>\varepsilon_0$ for $k=1,\dots,T-1$, then
    \begin{align}\label{eq_betamax}
\beta_{T-2} < \max \left\{ \frac{\Lipf + \varepsilon_1}{\sigmamin},
\dfrac{4(f(x_0)-\flow)}{R}\right\}\varepsilon_0\inv,
    \end{align}
    where $\Lipf \geq 0$ is such that 
    \begin{align}
    	\norm{\nabla f(x)} \leq \Lipf \quad \text{ for all } x\in \sublevelset.
    \end{align}
\end{lemma}
\begin{proof}
Note that $\Lipf\geq 0$ is well defined since $\sublevelset$ is
bounded~(\aref{assu_bounded_sublevelsets}). Suppose by contradiction that there
exists $k \in \{0,\dots,T-2\}$ such that $\norm{c(\xkp)}>\varepsilon_0$ and
	\begin{align}\label{eq_contradiction_beta}
\beta_k \geq \max \left\{ \frac{\Lipf + \varepsilon_1}{\sigmamin},
\dfrac{4(f(x_0)-\flow)}{R}\right\}\varepsilon_0\inv.
	\end{align}
	        The condition $Q_{\beta_k}(x_{k+1}) \leq Q_{\beta_k}(x_0)$ yields
\begin{align}
f(x_{k+1}) + \dfrac{\beta_k}{2}\norm{c(x\kplus)}^2 \leq f(x_0) +
\frac{\beta_k}{2}\norm{c(x_0)}^2.
\end{align}
Since $\varepsilon_0 <R$, condition~\eqref{eq_contradiction_beta} implies that
$\beta_k \geq 4(f(x_0) - \flow)/R^2$. Using $\norm{c(x_0)}^2 \leq
\frac{\varepsilon_0^2}{2}$, it follows
\begin{align*}
\norm{c(x\kplus)}^2 &\leq \frac{2(f(x_0) - \flow)}{\beta_k} + \norm{c(x_0)}^2 \\
&\leq \frac{R^2}{2} + \frac{R^2}{2} \leq R^2.
\end{align*}
Thus, $x_{k+1}$ belongs to the regular region $\LICQregion$, which ensures 
\begin{align*}
    \norm{J(x_{k+1})\T c(x_{k+1})} \geq \sigmamin \norm{c(\xkp)}.
\end{align*}
Notice that $\nabla Q_\beta(x) = \nabla f(x) + \beta J(x)\transpose c(x)$.
Therefore, the first-order condition~\eqref{eq:first_order_condition} gives
\begin{equation}
\begin{aligned}
    \tau(\xkp) & \geq \norm{\nabla Q_{\beta_k}(\xkp)}\\
    &= \norm{ \nabla f(\xkp) + \beta_k J(\xkp)\transpose c(\xkp)}\\
     &\geq \beta_k \norm{J(\xkp)\transpose c(\xkp)} - \norm{\nabla f(\xkp)}\\
         &\geq \beta_k \sigmamin \norm{c(\xkp)} - \norm{\nabla f(\xkp)},
        \end{aligned}
\end{equation}
which gives
\begin{align}\label{eq_bound_infeasibility_x_plus}
  \sigmamin \beta_k \norm{c(\xkp)} \leq \tau(\xkp) + \norm{\nabla f(\xkp)}.
\end{align}
By Lemma~\ref{lemma_sublevelset}, $x_{k+1} \in \sublevelset$, and therefore
$\norm{\nabla f(x_{k+1})} \leq \Lipf$. Since $\norm{c(\xkp)} >\varepsilon_0$, we have
$$\tau(\xkp) = \max\left \{ \varepsilon_1, \min\left( \taucap,
\frac{\varepsilon_1}{\varepsilon_0}\norm{c(\xkp)}\right)\right\} \leq
\max\left \{ \varepsilon_1,
\frac{\varepsilon_1}{\varepsilon_0}\norm{c(\xkp)}\right\}  =
\frac{\varepsilon_1}{\varepsilon_0}\norm{c(\xkp)}.$$
This gives 
\begin{align*}
\sigmamin \beta_k\norm{c(\xkp)} &\leq
\frac{\varepsilon_1}{\varepsilon_0}\norm{c(\xkp)} + \Lipf,
\end{align*}
that is, 
\begin{align*}
\norm{c(\xkp)} &\leq \dfrac{\Lipf}{\sigmamin \beta_k -
\frac{\varepsilon_1}{\varepsilon_0}}.
\end{align*}
Since $\beta_k \geq \frac{\Lipf + \varepsilon_1}{\sigmamin \varepsilon_0}$, we
have $\norm{c(\xkp)} \leq \varepsilon_0$, which is a contradiction.
\end{proof}

We now state an improved outer iteration bound for QPM under the PL condition on
the constraint violation.
\begin{theorem}[Outer iterations of QPM under the PL
condition]\label{thm:Tdoublehat}
Under~\aref{assu_smooth_problem},~\aref{assu_lower_bound_f},
~\aref{assu_x0_eps_feasible},~\aref{assu_LICQ},
~\aref{assu_bounded_sublevelsets}, let
\begin{align}
T(\varepsilon_0):= \inf \{ k\in \mathbb{N}^*\colon \norm{c(x_k)} \leq
\varepsilon_0\}.
\end{align}
Then, 
\begin{align}\label{eq_doublehat}
T(\varepsilon_0) \leq \Tdoublehat:= 2 + \log_\alpha \left( \max\left\{
\frac{\Lipf+\varepsilon_1}{\sigmamin}, \frac{4
(f(x_0)-\flow)}{R}\right\}\beta_0\inv \varepsilon_0\inv \right),
\end{align}
and the point $x_{ T(\varepsilon_0)}$ generated by
Algorithm~\ref{algo_quadratic_penalty} is an~\ref{eq:focp} point
of~\eqref{eq:P}, satisfying
\begin{align}
\norm{c \!\left(x_{T(\varepsilon_0)}\right)}&\leq \varepsilon_0  &&\textrm{ and } &
\norm{\nabla f \!\left(x\indexTlimit\right) + \beta_{\Tlimit-1} \sum_{i=1}^m
c_i\!\left(x\indexTlimit\right) \nabla c_i \!\left(x\indexTlimit\right) } \leq \varepsilon_1 .
\end{align}
\end{theorem}
\begin{proof}
By definition of $\Tlimit$, we have $\norm{c(x_k)}>\varepsilon_0$ for
$k=1,\dots,\Tlimit-1$. Lemma~\ref{lemma_tight_upper_bound_beta} gives
    \begin{align*}
\alpha^{\Tlimit-2} \beta_0  = \beta_{\Tlimit-2} < \betamax := \max \left\{
\frac{\Lipf+\varepsilon_1}{\sigmamin},
\dfrac{4(f(x_0)-\flow)}{R}\right\}\varepsilon_0\inv.
    \end{align*}
    Therefore, $\Tlimit$ is finite and in particular
    \begin{align}
        \alpha^{\Tlimit-2} \beta_0  = \beta_{\Tlimit-2} <\betamax.
    \end{align}
This gives 
    \begin{align}
\Tlimit \leq 2 + \log_\alpha\left(
\dfrac{\betamax}{\beta_0}\right)=:\Tdoublehat.
    \end{align}
Finally we have $\norm{c(x_{T(\varepsilon_0)})}\leq \varepsilon_0$ by
construction, which ensures that $x\indexTlimit$ is an~\ref{eq:focp} point with
multipliers $-\beta_{\Tlimit-1} c(x\indexTlimit)$ (see the proof of
Theorem~\ref{thm_outer_complexity_no_licq}).
\end{proof}

\begin{remark}
As the proof of Lemma~\ref{lemma_tight_upper_bound_beta} suggests, we can
replace the requirement $\norm{c(x_0)}\leq \varepsilon_0/\sqrt{2}$ by $x_0\in
\LICQregion$ and maintain convergence, provided that the norm of $\nabla f$
remains bounded over the iterates.
\end{remark}


\section{Total evaluation complexity with first-order inner
solver}\label{sec_total_complexity_FO}
In this section and the next, we derive total evaluation complexity bounds for QPM
(Algorithm~\ref{algo_quadratic_penalty}), which give an upper bound on the
number of oracle calls (evaluations of $f$, $c$ and their derivatives).
The inner evaluation complexity is measured in evaluations of
$Q_{\beta_k}$, $\nabla Q_{\beta_k}$, and $\nabla^2 Q_{\beta_k}$,
 with each such evaluation corresponding to a first- or second-order oracle call.

This section presents results where the inner solver of QPM is a
first-order method. We detail the results with and without the PL condition on
the constraint violation~(\aref{assu_LICQ}).

\subsection{Lemmas on first-order inner solvers}
We suppose that a first-order method---called $\calM_1$---minimizes
$Q_{\beta_k}$ in each subproblem of QPM. The following assumption on the
(inner) evaluation complexity is typical of first-order methods for
unconstrained nonconvex optimization.
\begin{assumption}\label{assu_A1}
Given a continuously differentiable function $F : \mathbb{R}^n \to \mathbb{R}$
with
\[
\mathcal{L}_F(\tilde{x}) = \{ z \in \mathbb{R}^n \mid F(z) \leq F(\tilde{x}) \}
\]  
bounded for some $\tilde{x} \in \mathbb{R}^n$. The method $\mathcal{A}_1$
starting from $\tilde{x}$ needs at most
\[
C_{\mathcal{A}_1} L_1 \big(F(\tilde{x}) - F_{\text{low}}\big)\,\varepsilon^{-2}
\]  
evaluations of $F$ and $\nabla F$ to generate a point satisfying $\norm{\nabla
F(x)}\leq \varepsilon$, where $F_{\text{low}}$ is a lower bound of $F$, the
positive constant $C_{\mathcal{A}_1}$ depends only on the method
$\mathcal{A}_1$, and $\nabla F$ is $L_1$-Lipschitz continuous on
$\conv(\mathcal{L}_F(F(\tilde{x})))$.
\end{assumption}
The traditional Gradient Descent Method with Armijo line search
satisfies~\aref{assu_A1}~\citep[Appendix 1]{grapiglia2023Worstcase}.

We make the following smoothness assumption.
\begin{assumption}\label{assu_twice_differentiable}
    The functions $f$ and $c$ are twice continuously differentiable.
\end{assumption}
Recall that $\phi(x) = \frac12 \norm{c(x)}^2$.
\begin{lemma}\label{lemma_lipschitz_continuity}
Under~\aref{assu_bounded_sublevelsets} and~\aref{assu_twice_differentiable},
there exists constants $\Lipgradf$ and $\LipgradP>0$ such that
\begin{align}\label{eq_lipschitz_constants_f_c}
\norm{\hess f(x)}&\leq \Lipgradf &&\textrm{ and } & \norm{\hess \phi(x)}\leq
\LipgradP
\end{align}
for all $x\in \conv \left( \calL_f(2f(x_0) - \flow)\right)$.

In addition, for any $\beta>0$, we have that $\nabla Q_\beta$ is $(\Lipgradf +
\beta \LipgradP)$-Lipschitz continuous on $\conv(\calL_f(2f(x_0) - \flow))$.
\end{lemma}
\begin{proof}
By~\aref{assu_bounded_sublevelsets}, the sublevel set $\calL_f(2f(x_0) -
\flow)$ is bounded, so its  convex hull is bounded.
By~\aref{assu_twice_differentiable}, the Hessians $\hess f$ and $\hess P$ are
continuous, hence there exists constants $\Lipgradf$ and $\LipgradP>0$ such
that~\eqref{eq_lipschitz_constants_f_c} holds. Thus, for all $x$ in
$\conv(\calL_f (2f(x_0) - \flow))$ and any $\beta>0$,
    \begin{align*}
\norm{\hess Q_\beta(x)} &\leq \norm{\hess f(x)} + \beta \norm{\hess \phi(x)}\\
        &\leq \Lipgradf + \beta \LipgradP.
    \end{align*}
Therefore, $\nabla Q_\beta$ is $(\Lipgradf+\beta \LipgradP)$-Lipschitz
continuous on $\conv\left(\sublevelset\right)$.
\end{proof}

The next lemma establishes the inner evaluation complexity of QPM with a
first-order solver.
\begin{lemma}[Inner evaluation complexity of $\calM_1$]
\label{inner_complexity_A1}
Under~\aref{assu_lower_bound_f},~\aref{assu_bounded_sublevelsets}
and~\aref{assu_twice_differentiable},
consider iteration $k$ of Algorithm~\ref{algo_quadratic_penalty}, 
where a monotone first-order method $\calM_1$ is
used to minimize $Q_{\beta_k}$ starting from
\begin{equation}\label{eq_x_tilde}
\tilde x_{k,0} := \underset{x\in \{x_k,x_0\}}{\argmin} Q_{\beta_k}(x).	
\end{equation}
If $\calM_1$ satisfies~\aref{assu_A1}, the method $\calM_1$ generates $\xkp$
satisfying the subproblem conditions~\eqref{eq:zeroth_order_condition}
and~\eqref{eq:first_order_condition}
in at most
\begin{align}
2C_{\calM_1}(\Lipgradf + \beta_k \LipgradP)(f(x_0)- \flow) \varepsilon_1^{-2}
\end{align}
first-order oracle calls, where $\Lipgradf$
and $\LipgradP$ are defined in Lemma~\ref{lemma_lipschitz_continuity}.
\end{lemma}
\begin{proof}
By~\aref{assu_lower_bound_f}, $Q_{\beta_k}$ is lower
bounded:
    \begin{align}
Q_{\beta_k}(x) = f(x) + \frac{\beta_k}{2}\norm{c(x)}^2_2 \geq f(x) \geq
\flow,\quad \text{ for all }x\in \Rn.
    \end{align}

Let $x\in \calL_{Q_{\beta_k}}(\tilde x_{k,0})$, since
$\norm{c(x_k)}>\varepsilon_0$, Lemma~\ref{lemma_loose_upper_bound_beta} gives
     \begin{align}\label{eq_Q_bound_f}
f(x) \leq Q_{\beta_k}(x) \leq Q_{\beta_k}(\tilde x_{k,0}) \leq Q_{\beta_k}(x_0)
&= f(x_0) + \frac{\beta_k}{2}\norm{c(x_0)}^2 \leq 2 f(x_0) - \flow,
     \end{align}
     which shows
$$\calL_{Q_{\beta_k}}(\tilde x_{k,0}) \subset \sublevelset.$$
Therefore, by Lemma~\ref{lemma_lipschitz_continuity}, $\nabla Q_{\beta_k}$ is
$(\Lipgradf + \beta_k \LipgradP)$-Lipschitz continuous on $\calL_{Q_{\beta_k}}
\left(\tilde x_{k,0}\right)$.
By~\aref{assu_A1}, since $\varepsilon_1 \leq \tau(x)$ for all $x\in \Rn$,
$\calM_1$ generates $\xkp$ satisfying $\norm{\nabla
Q_{\beta_k}(\xkp)}\leq \tau(\xkp)$ and $Q_{\beta_k}(\xkp)
\leq Q_{\beta_k}(\tilde x_{k,0})$ in at most
\begin{align}
C_{\calM_1}(\Lipgradf + \beta_k \LipgradP)(Q_{\beta_k}(\tilde x_{k,0})- \flow)
\varepsilon_1^{-2}
\end{align}
first-order oracle calls.
The conclusion follows from~\eqref{eq_Q_bound_f}. 
\end{proof}

\subsection{First-order inner solver without the PL condition}
Based on the inner evaluation complexity of Lemma~\ref{inner_complexity_A1}, we
have the following total evaluation complexity for
Algorithm~\ref{algo_quadratic_penalty}, without the PL condition on the
constraint violation.

\begin{theorem}[Total evaluation complexity with first-order inner solver
without PL]
\label{thm_A1_NO_LICQ}
Under~\aref{assu_lower_bound_f},
~\aref{assu_x0_eps_feasible},
~\aref{assu_bounded_sublevelsets},~\aref{assu_twice_differentiable}, suppose
that at each iteration of QPM (Algorithm~\ref{algo_quadratic_penalty}), the
point $x_{k+1}$ is computed using a monotone first-order method $\calM_1$
initialized at $$\tilde x_{k,0} = \argmin_{x\in \{x_k,x_0\}} Q_{\beta_k}(x).$$
If the method $\calM_1$ satisfies~\aref{assu_A1}, then QPM generates
an~\ref{eq:focp} point in at most
\begin{align}
8\alpha\Tsinglehat  C_{\calM_1}  \Deltaf^2 (\Lipgradf + \LipgradP)
\varepsilon_0^{-2}\varepsilon_1^{-2}
\end{align}
first-order oracle calls, where $\Lipgradf$
and $\LipgradP$ are defined in Lemma~\ref{lemma_lipschitz_continuity} and $\Tsinglehat$
is defined in~\eqref{eq:Tsinglehat}.
\end{theorem}
\begin{proof}
By Theorem~\ref{thm_outer_complexity_no_licq}, the number of outer iterations
$\Tlimit$ is upper bounded by $\Tsinglehat$. Since
$\norm{c(x_k)}>\varepsilon_0$ for $k = 1,\dots,T(\varepsilon_0)-1$,
Lemma~\ref{inner_complexity_A1} gives that the $k$th iteration of
Algorithm~\ref{algo_quadratic_penalty} requires at most
$$2C_{\calM_1}(\Lipgradf + \beta_k \LipgradP)(f(x_0)- \flow)
\varepsilon_1^{-2}$$
first-order oracle calls.
Lemma~\ref{lemma_loose_upper_bound_beta} implies
\(1\leq \beta_{\Tlimit	-1} \leq 4\alpha (f(x_0) - \flow) \varepsilon_0^{-2}\). It follows that the total number of first-order oracle
calls is bounded by
\begin{equation*}
\begin{aligned}
&\sum_{k=0}^{\Tlimit-1} 2C_{\calM_1}(\Lipgradf + \beta_k
\LipgradP)(f(x_0)-\flow)\varepsilon_1^{-2}\\
&\leq  \sum_{k=0}^{\Tlimit-1}   2C_{\calM_1} \beta_k (\Lipgradf +
\LipgradP)(f(x_0)- \flow) \varepsilon_1^{-2}\\
&\leq  \Tsinglehat 8\alpha C_{\calM_1}(\Lipgradf + \LipgradP)(f(x_0)- \flow)^2
\varepsilon_0^{-2}\varepsilon_1^{-2}.
\end{aligned}
\end{equation*}
\end{proof}

From Theorem~\ref{thm_A1_NO_LICQ}, in the absence of the PL condition on the
constraint violation, QPM equipped with a first-order inner solver requires at
most
$\mathcal{O}\!\left(\left|\log \!\left(\beta_{0}^{-1}\varepsilon_{0}^{-2}\right)\!\right|
\!\varepsilon_{0}^{-2
}\varepsilon_{1}^{-2}\right)$ calls to a first-order oracle to find
an~\ref{eq:focp} point. In particular, when
$\beta_{0}=\calO \!\left(\varepsilon_{0}^{-2}\right)$, the complexity bound reduces to
$\calO\!\left(\varepsilon_0^{-2}\varepsilon_1^{-2}\right)$.

\subsection{First-order inner solver under the PL condition}
The total evaluation complexity improves by a factor $\varepsilon_0\inv$ under
the PL condition on the constraint violation (\aref{assu_LICQ}).
\begin{theorem}[Total evaluation complexity with first-order inner solver under
PL]\label{thm_A1_LICQ}
Under \aref{assu_lower_bound_f},
\aref{assu_x0_eps_feasible}, \aref{assu_LICQ},
\aref{assu_bounded_sublevelsets}, \aref{assu_twice_differentiable}, suppose
that at each iteration of QPM
(Algorithm~\ref{algo_quadratic_penalty}), the point $x_{k+1}$ is computed using
a monotone first-order method $\calM_1$ initialized at $$\tilde x_{k,0} =
\argmin_{x\in \{x_k,x_0\}} Q_{\beta_k}(x).$$
If the method $\calM_1$ satisfies~\aref{assu_A1}, then QPM generates
an~\ref{eq:focp} point in at most
\begin{align}
2\alpha \Tdoublehat C_{\calM_1} (\Lipgradf + \LipgradP)(f(x_0)-\flow)\max
\left\{ \frac{\Lipf +\varepsilon_1}{\sigmamin} ,
\dfrac{4(f(x_0)-\flow)}{R}\right\}\varepsilon_0\inv\varepsilon_1^{-2}
\end{align}
first-order oracle calls, where $\Lipgradf$
and $\LipgradP$ are defined in Lemma~\ref{lemma_lipschitz_continuity},
$\sigmamin$ and $R$ are defined in~\aref{assu_LICQ}, $\Tdoublehat$ is defined
in~\eqref{eq_doublehat} and $\Lipf$ is defined in
Lemma~\ref{lemma_tight_upper_bound_beta}.
\end{theorem}
\begin{proof}
By Theorem~\ref{thm:Tdoublehat}, the number of outer iterations $\Tlimit$ is
upper bounded by $\Tdoublehat$.
Since $\norm{c(x_k)}>\varepsilon_0$ for $k = 1,\dots,T(\varepsilon_0)-1$, by
Lemma~\ref{inner_complexity_A1}, iteration $k$ of
Algorithm~\ref{algo_quadratic_penalty} requires at most
    $$
2C_{\calM_1}(\Lipgradf + \beta_k \LipgradP)(f(x_{0})- \flow)
\varepsilon_1^{-2}$$
first-order oracle calls.
Lemma~\ref{lemma_tight_upper_bound_beta} implies
   \begin{align}
1\leq \beta_{\Tlimit-1} < \alpha\max \left\{ \frac{\Lipf+ \varepsilon_1}{\sigmamin},
\dfrac{4(f(x_0)-\flow)}{R}\right\}\varepsilon_0\inv.
    \end{align}  
     Therefore, the total number of first-order oracle calls is bounded by
\begin{align}
&\sum_{k=0}^{\Tlimit-1} 2C_{\calM_1}(\Lipgradf + \beta_k
\LipgradP)(f(x_0)-\flow)\varepsilon_1^{-2}\\
&\le \sum_{k=0}^{\Tlimit-1} 2C_{\calM_1}(\Lipgradf +
\LipgradP)\beta_k(f(x_0)-\flow)\varepsilon_1^{-2}\\
&\leq \Tlimit 2C_{\calM_1} (\Lipgradf + \LipgradP)(f(x_0)-\flow)\alpha\max
\left\{ \frac{\Lipf+ \varepsilon_1}{\sigmamin},
\dfrac{4(f(x_0)-\flow)}{R}\right\}\varepsilon_0\inv\varepsilon_1^{-2} \\
&\leq 2\alpha \Tdoublehat C_{\calM_1} (\Lipgradf + \LipgradP)(f(x_0)-\flow)\max
\left\{ \frac{\Lipf+\varepsilon_1}{\sigmamin},
\dfrac{4(f(x_0)-\flow)}{R}\right\}\varepsilon_0\inv\varepsilon_1^{-2}.\qedhere
\end{align}
\end{proof}
From Theorem~\ref{thm_A1_LICQ}, under the PL condition on the constraint
violation~(\aref{assu_LICQ}), QPM equipped with a first-order inner solver
requires at most
$\mathcal{O}\left(\left|\log(\beta_{0}^{-1}\varepsilon_{0}^{-1})\right
|\varepsilon_{0}^{-1}\varepsilon_{1}^{-2}\right)$ calls to a first-order
oracle to find an~\ref{eq:focp} point. In particular, when
$\beta_{0}=\calO\!\left(\varepsilon_{0}^{-1}\right)$, the complexity bound
reduces to $\calO\!\left(\varepsilon_0^{-1}\varepsilon_1^{-2}\right)$.

\section{Total evaluation complexity with second-order inner
solver}\label{sec_total_complexity_SO}
We consider the total evaluation complexity of QPM when a second-order
method---called $\calM_2$---minimizes the subproblem. We detail the results with and without the PL
condition on the constraint violation.
 
\subsection{Lemmas on second-order inner solvers}

We consider the following assumptions on the problem and the (inner) evaluation
complexity of the second-order solver.
\begin{assumption}\label{assu_c3}
    The functions $f$ and $c$ are three times continuously differentiable.
\end{assumption}

\begin{assumption}\label{assu_A2}
Given a three times continuously differentiable function $F\colon \Rn \to \R$,
consider that the method $\calM_2$ is applied to minimize $F$ starting from
$\tilde x \in \Rn$, where \[\calL_F(\tilde x) = \{z\in \Rn| F(z) \leq F(\tilde
x)\}\]
is bounded. There exists an algorithmic constant of $\calM_2$, called
$\Delta_{\max}$, such that all trial points and iterates visited by $\calM_2$
belong to
$$\Omega_F = \lbrace x+d \in \Rn | x\in \calL_F(\tilde x) \text{ and } \norm{d}
\leq \Delta_{\max} \rbrace,$$
	and the method $\calM_2$ takes at most 
	$$C_{\calM_2}(F(\tilde x) - \Flow)L^{1/2} \varepsilon^{-3/2} $$
evaluations of $F$, $\nabla F$, and $\hess F$ to produce a point $x\in \Rn$
with $\norm{\nabla F(x)}\leq \varepsilon$, where $\hess F$ is $L$-Lipschitz
continuous on $\Omega_F$.
\end{assumption}
The development of second-order algorithms for unconstrained minimization with
optimal worst-case complexity is a very active field of research. For example,
the trust-region method from~\citep{hamad2025Simple}, once augmented with an
upper bound $\Delta_{\max}$ on the trust-region radius,
satisfies~\aref{assu_A2}.
Under~\aref{assu_bounded_sublevelsets} and~\aref{assu_A2}, we define the
bounded set
\begin{align}\label{eq_omega_star}
\Omegastar := \conv \lbrace x+d \in \Rn : x\in \sublevelset \text{ and }
\norm{d} \leq \Delta_{\max} \rbrace.
\end{align}
Additionally, under~\aref{assu_c3}, there exists finite constants
\begin{align}\label{eq_lipschitz_constants_2}
L_{f,2}\ &:=\ \underset{x\in \Omegastar}{\sup} \|\D^3 f(x)\|   &&\textrm{ and }
& L_{P,2}\  :=\ \underset{x\in \Omegastar}{\sup} \|\D^3 \phi(x)\| \, .
\end{align}
In the following lemma, we show that, for all $\beta\ge0$, the Hessian $\hess
Q_\beta$ is Lipschitz continuous on $\Omegastar$. We emphasize that we do not
assume Lipschitz continuity of some derivative of $f$ and $c$, merely that $f$
has a bounded sublevel set~(\aref{assu_bounded_sublevelsets}).
\begin{lemma}\label{lemma_lipschitz_continuity2}
Under~\aref{assu_lower_bound_f},~\aref{assu_bounded_sublevelsets},~\aref{assu_c3},~\aref{assu_A2}, for any $\beta>0$,
the Hessian $\hess Q_\beta$ is $(\Liphessf + \beta \LiphessP)$-Lipschitz
continuous on the set $\Omegastar$~\eqref{eq_omega_star}.
\end{lemma}
\begin{proof}
The set $\Omegastar$ defined in~\eqref{eq_omega_star} is bounded
by~\aref{assu_bounded_sublevelsets} and the constants $\Liphessf$ and
$\LiphessP$ in~\eqref{eq_lipschitz_constants_2} are well defined and finite. We
find that, for all $x\in \Omegastar$ and $\beta>0$,
\begin{equation*}
    \begin{aligned}
        \norm{\D^3 Q_\beta(x)} &\leq \norm{\D^3 f(x)+ \beta \D^3 \phi(x)}\\ 
        &\leq \norm{\D^3 f(x)} + \beta \norm{\D^3 \phi(x)}\\
        &\leq \Liphessf + \beta \LiphessP.
    \end{aligned}
\end{equation*}
Thus,  $\norm{\D^3 Q_\beta(x)} \leq \Liphessf + \beta \LiphessP$ for all $x\in
\Omegastar$, and therefore $\hess Q_\beta$ is $(\Liphessf+\beta
\LiphessP)$-Lipschitz continuous on $\Omegastar$.
\end{proof}

This allows to derive total evaluation complexity bounds for QPM. We begin
with an inner evaluation complexity for the second-order subproblem solver.

\begin{lemma}[Inner evaluation complexity of $\calM_2$]
\label{lemma_inner_iter_A2}
Under~\aref{assu_lower_bound_f},~\aref{assu_bounded_sublevelsets},~\aref{assu_c3},
consider
iteration $k$ of Algorithm~\ref{algo_quadratic_penalty} where a monotone
second-order method $\calM_2$ minimizes $Q_{\beta_k}$ with starting point
$$\tilde x_{k,0} = \argmin_{x\in \{x_k,x_0\}} Q_{\beta_k}(x).$$
If $\calM_2$ satisfies~\aref{assu_A2}, the method $\calM_2$ generates $\xkp$
satisfying the subproblem conditions~\eqref{eq:zeroth_order_condition}
and~\eqref{eq:first_order_condition} in at most
\begin{align}
2 C_{\calM_2}(f(x_0)- \flow)(\Liphessf + \beta_k \LiphessP)^{\frac{1}{2}}
\varepsilon_1^{-\frac{3}{2}}
   \end{align}
second-order oracle calls, where
$\Liphessf$ and $\LiphessP$ are defined in~\eqref{eq_lipschitz_constants_2}.
\end{lemma}
\begin{proof}
By~\aref{assu_lower_bound_f}, we have $Q_{\beta_k}(x) \geq \flow \quad \text{
for all }x\in \Rn$.
We also have that $\calL_{Q_{\beta_k}}(\tilde x_{k,0}) \subset
\sublevelset$~\eqref{eq_Q_bound_f}. Therefore, all iterates and trial points of
$\calM_2$ applied to $Q_{\beta_k}$ starting from $\tilde x_{k,0}$ remain in the
set $\Omegastar$	~\eqref{eq_omega_star}. Furthermore, the Hessian $\hess Q_{\beta_k}$ is $(\Liphessf +
\beta_k \LiphessP)$-Lipschitz continuous on $\Omegastar$ by
Lemma~\ref{lemma_lipschitz_continuity2}.
   
Therefore, since $\varepsilon_1 \leq \tau(x)$ for all $x\in \Rn$,
 $\calM_2$ generates $\xkp$ satisfying $\norm{\nabla
Q_{\beta_k}(\xkp)} \leq \tau(\xkp)$ and $Q_{\beta_k}(\xkp)
\leq Q_{\beta_k}(\tilde x_{k,0})$ in at most
\begin{align}
C_{\calM_2}	(Q_{\beta_k}(\tilde x_{k,0})- \flow) (\Liphessf + \beta_k
\LiphessP)^{\frac{1}{2}}  \varepsilon_1^{-\frac{3}{2}}
	\end{align}
second-order oracle calls. The
conclusion follows from $Q_{\beta_k}(\tilde x_{k,0}) \leq 2 f(x_0) -
\flow$~\eqref{eq_Q_bound_f}.
\end{proof}
Considering this result, we derive total evaluation complexity bounds for
Algorithm~\ref{algo_quadratic_penalty}.

\subsection{Second-order inner solver without the PL condition}
In this section, we give a total evaluation complexity bound for QPM with a second-order
inner solver and without the PL condition on the constraint violation.
\begin{theorem}[Total evaluation complexity with second-order inner solver
without PL]\label{thm_A2_no_LICQ}
Under~\aref{assu_lower_bound_f},
~\aref{assu_x0_eps_feasible},~\aref{assu_bounded_sublevelsets},
~\aref{assu_c3}, suppose that at
each iteration of QPM (Algorithm~\ref{algo_quadratic_penalty}), the point
$\xkp$ is computed using a monotone second-order method $\calM_2$ initialized
at $$\tilde x_{k,0} = \argmin_{x\in \{x_k,x_0\}} Q_{\beta_k}(x).$$
If the method $\calM_2$ satisfies~\aref{assu_A2}, then QPM generates
an~\ref{eq:focp} point in 	at most
	\begin{align}
4\Tsinglehat   C_{\calM_2}\sqrt{\alpha}(f(x_0)- \flow)^{\frac32}(\Liphessf +
\LiphessP)^{\frac12}   \varepsilon_0\inv  \varepsilon_1^{-\frac32}
	\end{align}
second-order oracle calls, where
$\Liphessf$ and $\LiphessP$ are defined in~\eqref{eq_lipschitz_constants_2},
and $\Tsinglehat$ is defined in~\eqref{eq:Tsinglehat}.
\end{theorem}
\begin{proof}
By Theorem~\ref{thm_outer_complexity_no_licq}, the number of outer iterations
$\Tlimit$ is upper bounded by $\Tsinglehat$. Since
$\norm{c(x_k)}>\varepsilon_0$ for $k = 1,\dots,T(\varepsilon_0)-1$,
Lemma~\ref{lemma_inner_iter_A2} gives that iteration $k$ of
Algorithm~\ref{algo_quadratic_penalty} requires at most
\begin{align}
2 C_{\calM_2}(f(x_0)- \flow)(\Liphessf + \beta_k \LiphessP)^{\frac{1}{2}}
\varepsilon_1^{-\frac{3}{2}}
   \end{align}
second-order	 oracle calls.
Lemma~\ref{lemma_loose_upper_bound_beta} gives
\(1\leq\beta_{\Tlimit-1} \leq 4\alpha (f(x_0) - \flow) \varepsilon_0^{-2}\). Therefore, the total number of second-order oracle calls
is bounded by
\begin{eqnarray*}
		\begin{aligned}
&\sum^{\Tlimit -1}_{k=0} 2 C_{\calM_2}(f(x_0)- \flow)(\Liphessf + \beta_k
\LiphessP)^{\frac{1}{2}} \varepsilon_1^{-\frac{3}{2}}\\
&\leq \sum^{\Tlimit -1}_{k=0} 2 C_{\calM_2}(f(x_0)- \flow)(\Liphessf +
\LiphessP)^{\frac{1}{2}} \beta_k^{\frac{1}{2}} \varepsilon_1^{-\frac{3}{2}}\\
&\leq 2\Tlimit   C_{\calM_2}(f(x_0)- \flow)(\Liphessf +
\LiphessP)^{\frac{1}{2}} (4\alpha (f(x_0) - \flow)
\varepsilon_0^{-2})^{\frac{1}{2}}  \varepsilon_1^{-\frac{3}{2}}\\
&\leq 4\Tsinglehat   C_{\calM_2}\sqrt{\alpha}(f(x_0)- \flow)^{\frac32}(\Liphessf +
\LiphessP)^{\frac{1}{2}}   \varepsilon_0\inv  \varepsilon_1^{-\frac{3}{2}}.
	\end{aligned}
\end{eqnarray*}
\end{proof}

From Theorem~\ref{thm_A2_no_LICQ}, in the absence of the PL condition on the
constraint violation, QPM equipped with a second-order inner solver requires at
most
$\mathcal{O}\left(|\log(\beta_{0}^{-1}\varepsilon_{0}^{-2})|\varepsilon_{0}^{-1
}\varepsilon_{1}^{-3/2}\right)$ calls to a second-order oracle to find
an~\ref{eq:focp} point. In particular, when $\beta_{0}=\calO \!
\left(\varepsilon_{0}^{-2}\right)$, the complexity bound reduces to
$\calO\!\left(\varepsilon_0^{-1}\varepsilon_1^{-3/2}\right)$.

\subsection{Second-order inner solver under the PL condition}
We now show a total evaluation complexity bound under the PL condition on the
constraint violation with a second-order solver in the subproblems.
\begin{theorem}[Total evaluation complexity with second-order inner solver
under PL]\label{thm_A2_LICQ}
Under~\aref{assu_lower_bound_f},
~\aref{assu_x0_eps_feasible},~\aref{assu_LICQ},
~\aref{assu_bounded_sublevelsets},
~\aref{assu_c3},
suppose that at each iteration of QPM (Algorithm~\ref{algo_quadratic_penalty}),
the point $\xkp$ is computed using a monotone second-order method $\calM_2$
initialized at
$$\tilde x_{k,0} = \argmin_{x\in \{x_k,x_0\}} Q_{\beta_k}(x).$$
If the method $\calM_2$ satisfies~\aref{assu_A2}, then QPM generates
an~\ref{eq:focp} in at most
	\begin{align}
2\Tdoublehat   C_{\calM_2}(f(x_0)- \flow) (\Liphessf + \LiphessP)^{\frac{1}{2}}
\sqrt{\alpha}\max \left\{ \frac{\Lipf+\varepsilon_1}{\sigmamin},
\dfrac{4(f(x_0)-\flow)}{R}\right\}^{\frac{1}{2}}
\varepsilon_0\invv{\frac{1}{2}} \varepsilon_1^{-\frac{3}{2}}
	\end{align}
second-order oracle calls, where $\Liphessf$ and $\LiphessP$
are defined in~\eqref{eq_lipschitz_constants_2}, $\sigmamin$ and $R$ are
defined in~\aref{assu_LICQ}, $\Tdoublehat$ is defined in~\eqref{eq_doublehat},
and $\Lipf$ is defined in
Lemma~\ref{lemma_tight_upper_bound_beta}.
\end{theorem}
\begin{proof}
By Theorem~\ref{thm:Tdoublehat}, the number of outer iterations $\Tlimit$ is
upper bounded by $\Tdoublehat$. Lemma~\ref{lemma_tight_upper_bound_beta} gives
   \begin{align}
1\leq   \beta_{\Tlimit-1} < \alpha\max \left\{ \frac{\Lipf+\varepsilon_1}{\sigmamin},
\dfrac{4(f(x_0)-\flow)}{R}\right\}\varepsilon_0\inv.
    \end{align}  
Lemma~\ref{lemma_inner_iter_A2} gives that iteration $k$ of
Algorithm~\ref{algo_quadratic_penalty} requires at most
\begin{align}
2 C_{\calM_2}(f(x_0)- \flow)(\Liphessf + \beta_k \LiphessP)^{\frac{1}{2}}
\varepsilon_1^{-\frac{3}{2}}
   \end{align}
second-order oracle calls. Therefore,
the total number of second-order oracle calls is bounded by
		\begin{align*}
&\sum^{\Tlimit -1}_{k=0} 2 C_{\calM_2}(f(x_0)- \flow)(\Liphessf + \beta_k
\LiphessP)^{\frac12} \varepsilon_1^{-\frac32}\\
&\leq \sum^{\Tlimit -1}_{k=0} 2 C_{\calM_2}(f(x_0)- \flow)(\Liphessf +
\LiphessP)^{\frac{1}{2}}  \beta_k^{\frac{1}{2}}\varepsilon_1^{-\frac{3}{2}}\\
&\leq 2\Tdoublehat   C_{\calM_2}(f(x_0)- \flow) (\Liphessf +
\LiphessP)^{\frac{1}{2}}  \sqrt{\alpha}\max \left\{ \frac{\Lipf +
\varepsilon_1}{\sigmamin}, \dfrac{4(f(x_0)-\flow)}{R}\right\}^{\frac{1}{2}}
\varepsilon_0\invv{\frac{1}{2}} \varepsilon_1^{-\frac{3}{2}}.\qedhere
	\end{align*}
\end{proof}
From Theorem~\ref{thm_A2_LICQ}, under the PL condition on the constraint
violation~(\aref{assu_LICQ}), QPM equipped with a second-order inner solver
requires at most
$\mathcal{O}\left(\left|\log(\beta_{0}^{-1}\varepsilon_{0}^{-1})\right
|\varepsilon_{0}^{-1/2}\varepsilon_{1}^{-3/2}\right)$ calls to a second-order
oracle to find an~\ref{eq:focp} point. In particular, when
$\beta_{0}=\calO\!\left(\varepsilon_{0}^{-1}\right)$, the complexity bound
reduces to $\calO\!\left(\varepsilon_0^{-1/2}\varepsilon_1^{-3/2}\right)$.


\section{Illustrative numerical results}
\label{sec_numerics}
In this section, we illustrate our theoretical findings. We demonstrate
numerically the gain in performance induced by the feasibility-aware tolerance
in the subproblems, and we also compare the performance of first- and
second-order solvers in the subproblems.

Our experiments use a Julia implementation of
Algorithm~\ref{algo_quadratic_penalty}. The code to reproduce the results is
available at \url{www.github.com/flgoyens/QPM}. Our implementation includes
first- and second-order methods for the subproblems. The first-order method,
called QPM-GD, uses a gradient descent method with an Armijo linesearch for the
inner minimizations. The second-order method, called QPM-TR, uses a trust
region method with exact Hessian for the inner minimizations. The trust-region
subproblem is solved via the truncated conjugate gradient
method~\citep{conn2000trust}. In our experiments, we use the values $\alpha =
1.2$ and $\beta_0=1$ unless specified otherwise.

In the experiments below, we compare the performance of two choices for the
subproblem tolerance: the non-adaptive rule
 \begin{align}\label{eq:tolerance_eps1}
	\norm{\nabla Q_{\beta_k}(x_{k+1})} \leq \varepsilon_1,
\end{align}
against the adaptive tolerance
 \begin{align}\label{eq:tau_k}
\norm{\nabla Q_{\beta_k}(x_{k+1})} \leq \tau_k :=
\max\left(\varepsilon_1,\dfrac{\varepsilon_1}{\varepsilon_0}\norm{c(\xkp)
}\right),
\end{align}
where $\tau_k$ is used as a label in our plots to denote the adaptive tolerance at
iteration $k$.

Our test problem is the minimization of the extended Rosenbrock function over
the unit sphere: for $n$ even,
\begin{equation}\label{eq:rosenbrock}
\begin{aligned}
& \underset{x\in \Rn}{\minimize}
& & \sum_{i=1}^{n/2} \left[100 (x_{2i}- x^2_{2i-1})^2 + (1- x_{2i-1})^2\right]
\\
& \text{subject to}
& & \norm{x}^2=1.
\end{aligned}
\end{equation}
The problem is smooth~\aref{assu_c3}, the cost function is lower
bounded~\aref{assu_lower_bound_f}, coercive~\aref{assu_bounded_sublevelsets},
and the constraint function satisfies~\aref{assu_LICQ} as a particular case of
Example~\ref{example:stiefel} with $p=1$.
The starting point is chosen as
\[
x_0 = \sqrt{\frac1n\left(1 + \dfrac{\varepsilon_0}{\sqrt{2}}\right)}\,
\mathbf{1}_n,
\]
in order to ensure $\norm{c(x_0)}= \varepsilon_0/\sqrt{2}$
(\aref{assu_x0_eps_feasible}).

\subsection*{First-order inner minimization: QPM-GD}
 
Figure~\ref{fig:rosenbrock1000} shows two variants of QPM-GD applied to
Problem~\eqref{eq:rosenbrock} in dimension $n=1000$. The blue curve represents
the adaptive tolerance~\eqref{eq:tau_k} in the subproblem (labelled
$\tau_k$), and the orange curve represent the constant 
tolerance~\eqref{eq:tolerance_eps1} (labelled $\varepsilon_1$).  It appears clearly that
the adaptive tolerance performs less inner iterations in each subproblem. To
reach an accuracy of $10\invv{3}$, the adaptive QPM uses $1570$ gradients steps
and the non-adaptive QPM uses $3259$ gradient steps.
Table~\ref{tab:run_rosenbrock_gd_tau} shows that both methods terminate with
the same accuracy when the tolerance is set to $10^{-6}$, and reports the
markedly smaller number of oracle calls for the adaptive version.

\begin{figure}[ht]
\centering
\includegraphics[scale=0.5]{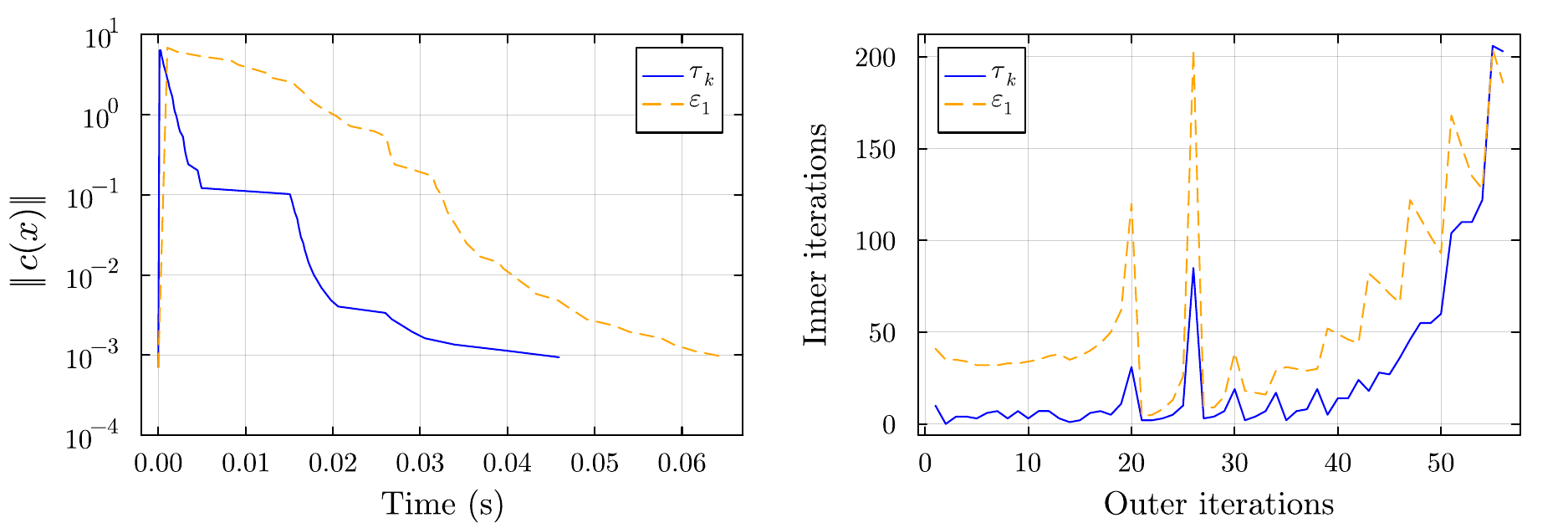}
\caption{Comparing QPM-GD with adaptive and fixed subproblem tolerance on
Problem~\eqref{eq:rosenbrock} with  $\beta_0 = 1$, $\alpha = 1.2$, $n=10^3$,
$\varepsilon_0 = \varepsilon_1 = 10\invv{3}$. Total number of inner iterations:
$1570$ for $\tau_k$ and $3259$ for $\varepsilon_1$.}
  \label{fig:rosenbrock1000}
\end{figure}

\begin{table}[ht] 
\centering
\begin{tabular}{l r  r r r r}
Method & $\norm{c(x_{\mathrm{final}})}$  & $f(x_{\mathrm{final}})$ &
$Q_{\beta}$ eval & $\nabla Q_{\beta}$ eval  \\
\hline
QPM-GD ($\tau_k$) & $7.1 \times 10^{-7}$  & $515.76$ & 8441 & 4583 \\
QPM-GD ($\varepsilon_1$) & $7.1 \times 10^{-7}$  & 515.76 & 12079 & 7771 \\
\end{tabular}
\caption{QPM-GD with $\tau_k$ vs $\varepsilon_1$ on
Problem~\eqref{eq:rosenbrock} with $n = 1000$, $\varepsilon_0 = \varepsilon_1 =
10\invv{6}$.}
\label{tab:run_rosenbrock_gd_tau}
\end{table}


\subsection*{Second-order inner minimization: QPM-TR}

Figure~\ref{fig_rosenbrock1000GDTR} shows the performance of QPM-GD and QPM-TR
on the same instance of Problem~\eqref{eq:rosenbrock} with $n=1000$, both use the
adaptive tolerance in the subproblems. It is apparent that QPM-TR takes much
less time to solve the problem. Additionally,
Table~\ref{tab:run_rosenbrock_gd_tr} shows that the second-order method
terminates with a smaller function value than the first-order method.

\begin{figure}[ht] 
\centering
\includegraphics[scale=0.5]{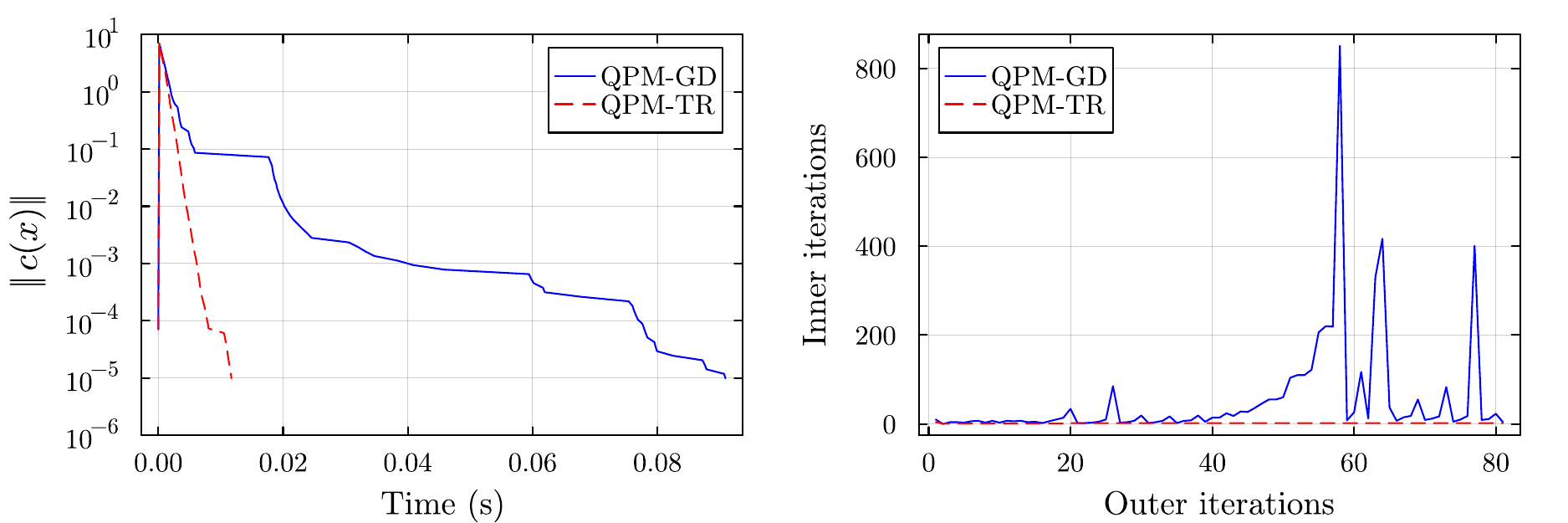}
\caption{Comparing QPM-GD and QPM-TR, both with adaptive subproblem tolerance
on Problem~\eqref{eq:rosenbrock} with $\beta_0 = 1$, $\alpha = 1.2$, $n=10^3$,
$\varepsilon_0 = \varepsilon_1 = 10\invv{5}$. Total number of inner iterations:
$4316$ for QPM-GD and $146$ for QPM-TR.}
  \label{fig_rosenbrock1000GDTR}
\end{figure}

\begin{table}[ht]
\centering
\begin{tabular}{l r r r r r r}
Method & $\norm{c(x_{\mathrm{final}})}$ & $f(x_{\mathrm{final}})$ & $Q_{\beta}$
eval & $\nabla Q_{\beta}$ eval & $\nabla^2 Q_{\beta}$ eval \\
\hline
QPM-GD ($\tau_k$) & $7.1 \times 10^{-7}$  & 515.76 & 8441 & 4583 & 0 \\
QPM-TR ($\tau_k$) & $9.2 \times 10^{-7}$  & 456.31 & 548 & 265 & 262 \\
\end{tabular}
\caption{QPM-GD vs QPM-TR on Problem~\eqref{eq:rosenbrock} with $n = 1000$,
$\varepsilon_0 = \varepsilon_1 = 10\invv{6}$ and $\tau(x) =  \tau_k$.}
\label{tab:run_rosenbrock_gd_tr}
\end{table}

\FloatBarrier

\section*{Conclusions}

In this work, we analyzed the worst-case oracle complexity of the Quadratic
Penalty Method (QPM) for smooth, nonconvex, equality-constrained optimization
problems, both with and without the PL condition on the constraint violation.
In the absence of PL, we established complexity bounds of
$\tilde{\calO}(\varepsilon_{0}^{-2}\varepsilon_{1}^{-2})$ and
$\tilde{\mathcal{O}}(\varepsilon_{0}^{-1}\varepsilon_{1}^{-3/2})$ for
obtaining~\ref{eq:focp} points when QPM is equipped with suitable first- and
second-order inner solvers, respectively.
Under the PL condition on the constraint violation, these bounds improve to
$\tilde{\calO}(\varepsilon_{0}^{-1}\varepsilon_{1}^{-2})$ and
$\tilde{\calO}(\varepsilon_{0}^{-1/2}\varepsilon_{1}^{-3/2})$, reflecting the
sharper dependence on feasibility accuracy afforded by the regularity
assumption. In both regimes, we further showed that the logarithmic dependence
on $\varepsilon_{0}$ vanishes when the initial penalty parameter is chosen
proportional to a suitable power of $\varepsilon_{0}^{-1}$. Our analysis
accommodates variants of QPM that use relaxed stopping criteria for the
subproblems. Leveraging this flexibility, we proposed a feasibility-aware
stopping rule that adaptively loosens the stationarity accuracy when far from
feasibility. This criterion preserves all theoretical guarantees and can yield
substantial practical speedups, as illustrated in our preliminary numerical
experiments.

\bibliographystyle{apalike} 
\bibliography{QPMLibrary.bib}   

\end{document}